\documentclass[12pt]{article}
\usepackage{amssymb,amsthm,amsmath,amsfonts,latexsym,tikz,hyperref,shuffle}
\usepackage[hmargin=1in,vmargin=1in]{geometry}

\newtheorem{thm}{Theorem}[section]
\newtheorem{prop}[thm]{Proposition}
\newtheorem{cor}[thm]{Corollary}
\newtheorem{lem}[thm]{Lemma}
\newtheorem{conj}[thm]{Conjecture}
\newtheorem{exa}[thm]{Example}
\newtheorem{question}[thm]{Question}

\newcommand{\shu}{\shuffle}

\DeclareMathOperator{\St}{St}
\DeclareMathOperator{\Pk}{Pk}
\DeclareMathOperator{\pk}{pk}
\DeclareMathOperator{\QSym}{QSym}
\DeclareMathOperator{\cDes}{cDes}
\DeclareMathOperator{\cdes}{cdes}
\DeclareMathOperator{\cPk}{cPk}
\DeclareMathOperator{\cpk}{cpk}
\DeclareMathOperator{\cSt}{cSt}
\DeclareMathOperator{\std}{std}
\DeclareMathOperator{\bru}{bru}
\DeclareMathOperator{\cbru}{cbru}

\newcommand{\ben}{\begin{enumerate}}
\newcommand{\een}{\end{enumerate}}
\newcommand{\ble}{\begin{lem}}
\newcommand{\ele}{\end{lem}}
\newcommand{\bth}{\begin{thm}}
\renewcommand{\eth}{\end{thm}}
\newcommand{\bpr}{\begin{prop}}
\newcommand{\epr}{\end{prop}}
\newcommand{\bco}{\begin{cor}}
\newcommand{\eco}{\end{cor}}
\newcommand{\bcon}{\begin{conj}}
\newcommand{\econ}{\end{conj}}
\newcommand{\bde}{\begin{defn}}
\newcommand{\ede}{\end{defn}}
\newcommand{\bex}{\begin{exa}}
\newcommand{\eex}{\end{exa}}
\newcommand{\barr}{\begin{array}}
\newcommand{\earr}{\end{array}}
\newcommand{\btab}{\begin{tabular}}
\newcommand{\etab}{\end{tabular}}
\newcommand{\beq}{\begin{equation}}
\newcommand{\eeq}{\end{equation}}
\newcommand{\bea}{\begin{eqnarray*}}
\newcommand{\eea}{\end{eqnarray*}}
\newcommand{\bal}{\begin{align*}}
\newcommand{\bce}{\begin{center}}
\newcommand{\ece}{\end{center}}
\newcommand{\bpi}{\begin{picture}}
\newcommand{\epi}{\end{picture}}
\newcommand{\bpp}{\begin{picture}}
\newcommand{\epp}{\end{picture}}
\newcommand{\bfi}{\begin{figure} \begin{center}}
\newcommand{\efi}{\end{center} \end{figure}}
\newcommand{\bprf}{\begin{proof}}
\newcommand{\eprf}{\end{proof}\medskip}

\newcommand{\bsl}{\begin{slide}{}}
\newcommand{\esl}{\end{slide}}
\newcommand{\bfr}{\begin{frame}}
\newcommand{\efr}{\end{frame}}

\newcommand{\hqed}{\hfill \qed}

\newcommand{\hso}[1]{\hspace{-1pt}}

\newcommand{\qmq}[1]{\quad\mbox{#1}\quad}

\newcommand{\emp}{\emptyset}

\newcommand{\sbs}{\subset}
\newcommand{\sbe}{\subseteq}

\newcommand{\gauss}[2]{\genfrac{[}{]}{0pt}{}{#1}{#2}}

\def\<{\langle}
\def\>{\rangle}

\newcommand{\ra}{\rightarrow}

\newcommand{\si}{\sigma}
\renewcommand{\th}{\theta}
\newcommand{\bbN}{{\mathbb N}}
\newcommand{\bbP}{{\mathbb P}}

\DeclareMathOperator{\Av}{Av}

\DeclareMathOperator{\des}{des}
\DeclareMathOperator{\Des}{Des}

\DeclareMathOperator{\maj}{maj}

\DeclareMathOperator{\Mod}{mod}

\DeclareMathOperator{\st}{st}

\begin{document}
\pagestyle{plain}

\title{Cyclic shuffle compatibility
}
\author{Rachel Domagalski\\[-5pt]
\small Department of Mathematics, Michigan State University,\\[-5pt]
\small East Lansing, MI 48824-1027, USA, {\tt domagal9@msu.edu}\\
Jinting Liang\\[-5pt]
\small Department of Mathematics, Michigan State University,\\[-5pt]
\small East Lansing, MI 48824-1027, USA, {\tt liangj26@msu.edu}\\
Quinn Minnich\\[-5pt]
\small Department of Mathematics, Michigan State University,\\[-5pt]
\small East Lansing, MI 48824-1027, USA, {\tt minnichq@msu.edu}\\
Bruce E. Sagan\\[-5pt]
\small Department of Mathematics, Michigan State University,\\[-5pt]
\small East Lansing, MI 48824-1027, USA, {\tt bsagan@msu.edu}\\
Jamie Schmidt\\[-5pt]
\small Department of Mathematics, Michigan State University,\\[-5pt]
\small East Lansing, MI 48824-1027, USA, {\tt schmi710@msu.edu}\\
Alexander Sietsema\\[-5pt]
\small Department of Mathematics, Michigan State University,\\[-5pt]
\small East Lansing, MI 48824-1027, USA, {\tt sietsem6@msu.edu}
}

\date{\today\\[10pt]
	\begin{flushleft}
	                                       \\[5pt]
									   %     https://mathscinet.ams.org/mathscinet/msc/msc2020.html
	\end{flushleft}}

\maketitle

\begin{abstract}
Consider a permutation $\pi$ to be any finite list of distinct positive integers.  
A statistic is a function $\St$ whose domain is all permutations.
Let $\pi\shu\si$ be the set of shuffles of two disjoint permutations $\pi$  and $\si$.
We say that $\St$ is shuffle compatible if the distribution of $\St$ over $\pi\shu\si$ depends only on $\St(\pi)$, $\St(\si)$, and the lengths of $\pi$ and $\si$.  This notion is implicit in Stanley's work on $P$-partitions and was first explicitly studied by Gessel and Zhuang.
One of the places where shuffles are useful is in describing the product in the algebra of quasisymmetric functions.  Recently Adin, Gessel, Reiner, and Roichman defined an algebra of cyclic quasisymmetric functions where a cyclic version of shuffling comes into play.  The purpose of this paper is to define and study cyclic shuffle compatibility.  In particular, we show how one can lift shuffle compatibility results for (linear) permutations to cyclic ones.  We then apply this result to cyclic descents and cyclic peaks.  We also discuss the problem of finding a cyclic analogue of the major index.
\end{abstract}

\section{Introduction}

Let $\bbN$ and $\bbP$ be the nonnegative and positive integers, respectively.  If $n\in\bbN$ then let $[n]=\{1,2,\ldots,n\}$.  Given a finite $A\sbe\bbP$ then a {\em linear permutation} of $A$ is a linear arrangement $\pi=\pi_1\pi_2\ldots\pi_n$ of the elements of $A$.  We let 
$$
L(A) = \{\pi \mid \text{$\pi$ is a linear permutation of $A$}\}.
$$
We often drop ``linear" if it is understood from context.  For example,
$$
L(\{1,3,6\})=\{ 136, 163, 316, 361, 613, 631\}.
$$
If $\pi=\pi_1\pi_2\ldots\pi_n$ then $n$ is called the {\em length} of $\pi$, written $\#\pi=|\pi|=n$.  The hash symbol and absolute value sign will also be used for the cardinality of a set.  In fact, we will often treat permutations as sets if no problem will result and write things such as $x\in\pi$ in place of the more cumbersome $\pi_i=x$ for some $i$.

A {\em (linear) permutation statistic} is a function $\St$ whose domain is all linear permutations.  There are four statistics which will concern us here.  A permutation $\pi=\pi_1\pi_2\ldots\pi_n$
has {\em Descent set}
$$
\Des\pi=\{i \mid \pi_i>\pi_{i+1}\},
$$
{\em descent number}
$$
\des\pi=\#\Des\pi,
$$
{\em Peak set}
$$
\Pk\pi=\{i \mid \pi_{i-1}<\pi_i>\pi_{i+1}\},
$$
and {\em peak number}
$$
\pk\pi =\#\Pk\pi.
$$
To illustrate, if $\pi=4218596$ then $\Des\pi=\{1,2,4,6\}$,
$\des\pi=4$, $\Pk\pi=\{4,6\}$, and $\pk\pi=2$.  We will be particularly concerned with a particular type of statistic which is determined by the descent set in the following sense.  Statistic $\St$ is a {\em descent statistic} if  $\Des\pi=\Des\si$ and $\#\pi=\#\si$ implies $\St\pi=\St\si$.  Note that $\Des$ itself, $\des$, $\Pk$, and $\pk$ are all descent statistics.  A statistic which is not a descent statistic would be $\St\pi=\pi_1$ for nonempty $\pi$ and $\St\emp=0$.  We will often need to evaluate a statistic $\St$ on a set of permutations $\Pi$.  So we define the {\em distribution} of $\St$ over $\Pi$ to be
$$
\St\Pi=\{\{\St\pi \mid \pi\in\Pi\}\}.
$$
Note that this is a set with multiplicity since $\St$ may evaluate to the same thing on various members of $\Pi$.  We will sometimes use exponents to denote multiplicities where, as usual, no exponent means multiplicity $1$.  For example
$$
\des\{123, 132, 312, 321\} =
\{\{0, 1, 1, 2\}\} = \{\{ 0, 1^2, 2\}\}.
$$
We will also need to combine functions on permutations and statistics.  If $\Pi$ and $\Pi'$ are sets of permutations and $\St$ is a statistic, then a function $f:\Pi\ra\Pi'$ is {\em $\St$-preserving} if $$
\St f(\pi) = \St\pi
$$
for all $\pi\in\Pi$.

If $\pi\in L(A)$ and $\si\in L(B)$ where $A\cap B=\emp$ then these permutations have {\em shuffle set}
$$
\pi\shu\si =\{ \tau\in  L(A\uplus B) \mid
\text{$\pi,\si$ are subwords of $\tau$}\}
$$
where a subwords of a permutation $\tau$  is a subsequence of (not necessarily consecutive) elements of $\tau$.  For example,
\begin{equation}
\label{25s73}
 25\shu 73 = \{2573, 2753, 2735, 7253, 7235, 7325\}.   
\end{equation}
Whenever we write the shuffle of two permutations, we will tacitly assume that they are from disjoint sets.

An important use of the shuffle set is in computing the product of two fundamental quasisymmetric functions in the algebra $\QSym$ of quasisymmetric functions.  More information about $\QSym$ can be found in the texts of Luoto, Mykytiuk, and van Willigenburg~\cite{lmv:iqs}, Sagan~\cite{sag:aoc}, or Stanley~\cite{sta:ec2}.  In order to prove that this product formula is well defined, one needs shuffle compatibility.  Roughly speaking, a statistic $\St$ is shuffle compatible if the distribution of $\St$ over $\pi\shu\si$ depends only on $\St\pi$, $\St\si$, and the lengths $\#\pi$ and $\#\si$.  To be precise, $\St$ is {\em shuffle compatible} if for all quadruples $\pi,\pi',\si,\si'$ with $\St\pi=\St\pi'$, $\St\si=\St\si'$, $\#\pi=\#\pi'$, and $\#\si=\#\si'$ we have
$$
\St(\pi\shu\si) = \St(\pi'\shu\si').
$$
For example, from equation~\eqref{25s73} it is easy to see that
$$
\des(25\shu 73)=\{\{1^3,2^3\}\}
$$
The reader can check that this is also the distribution $\des(12\shu 43)$ which is because $\des$ is shuffle compatible.
Shuffle compatibility is implicit in Stanley's theory of $P$-partitions~\cite{sta:osp} and he proved a result implying the shuffle compatibility of $\des$.  Gessel and Zhuang~\cite{gz:sps} were the first to define the concept explicitly and prove many shuffle compatibility results, including those for the other three statistics defined previously.
\begin{thm}
The statistics
$$
\Des, \des, \Pk, \pk
$$
are all shuffle compatible.\hqed
\end{thm}
\noindent Other work on shuffle compatibility has been done by 
Baker-Jarvis and Sagan~\cite{bs:bps}, by O{\u{g}}uz~\cite{ogu:ces}, and by
Grinberg~\cite{gri:sps}.

Recently Adin, Gessel, Reiner, and Roichman~\cite{agrr:cqf} introduced a cyclic version of quasisymmetric functions with a corresponding cyclic shuffle operation.  A linear permutation $\pi=\pi_1\pi_2\ldots\pi_n$ has a corresponding {\em cyclic permutation} $[\pi]$ which is the set of all its rotations, namely
$$
[\pi]=\{\pi_1\pi_2\ldots,\pi_n,\
\pi_2\ldots,\pi_n\pi_1,\
\ldots,\
\pi_n\pi_1\ldots\pi_{n-1}\}.
$$
For example
$$
[3725]=\{3725, 7253, 2537, 5372\}
$$
so that 
$$
[3725]=[7253]=[2537]=[5372].
$$
The reader should bear in mind the difference between the notation $[n]$ for an interval of integers and $[\pi]$ for a cyclic permutation.
We let
$$
C(A) = \{[\pi] \mid \text{$[\pi]$ is a cyclic permutation of $A$}\}.
$$
To illustrate
$$
C(\{2,3,5,7\})=\{[2357], [2375], [2537], [2735], [2573], [2753]\}.
$$

We can lift the linear permutation statistics we have introduced to the cyclic realm as follows.  Define the {\em cyclic descent set} and {\em cyclic descent number} for a linear permutation $\pi=\pi_1\pi_2\ldots\pi_n$ to be
$$
\cDes\pi =\{i\in[n] \mid 
\text{$\pi_i>\pi_{i+1}$ where $i$ is taken modulo $n$}\}
$$
and
$$
\cdes\pi=\#\cDes\pi,
$$
respectively.
Returning to $\pi=4218596$ we have $\cDes\pi=\{1,2,4,6,7\}$ and $\cdes\pi=5$.  Now  a cyclic permutation $[\pi]$ will have {\em cyclic descent set}
$$
\cDes[\pi]=\{\{\cDes\si \mid \si\in[\pi]\}\}.
$$
Notice that this is a multiset since two different $\si$ can have the same cyclic descent set.  Also define the {\em cyclic descent number}
$$
\cdes[\pi]=\cdes\pi.
$$
Note that this is well defined since all elements of $[\pi]$ have the same number of cyclic descents.  To illustrate
$$
\cDes[3725] =\{\{\cDes 3725, \cDes 7253, \cDes 2537, \cDes 5372\}\}
=\{\{\ \{1,3\}^2,\ \{2,4\}^2\ \}\}
$$
and $\cdes[3725]=2$.  An important point later will be that if $S$ is one of the sets in $\cDes[\pi]$ then all of the other member sets can be obtained by adding $i$ to the elements of $S$ modulo $n=\#\pi$ as $i$ runs over $[n]$. 

We deal with peaks in a similar manner, defining the {\em cyclic peak set} and {\em cyclic peak number} in the linear case by
$$
\cPk\pi =\{i\in[n] \mid 
\text{$\pi_{i-1}<\pi_i>\pi_{i+1}$ where $i$ is taken modulo $n$}\}
$$
and
$$
\cpk\pi=\#\cpk\pi,
$$
respectively.  For $\pi=4218593$ we have $\cPk\pi=\{1,4,6\}$ and $\cpk\pi=3$.  The extension to cyclic permutations is as expected
$$
\cPk[\pi]=\{\{\cPk\si \mid \si\in[\pi]\}\}
$$
and
$$
\cpk[\pi]=\cpk\pi.
$$
In general, a {\em cyclic permutation statistic} is any function $\cSt$ whose domain is cyclic permutations.  And $\cSt$ is a {\em cyclic descent statistic}
if  $\cDes[\pi]=\cDes[\si]$ and $\#\pi=\#\si$ implies $\cSt[\pi]=\cSt[\si]$.  All four of our cyclic statistics are cyclic descent statistics.
Functions preserving cyclic permutation statistics are defined in the obvious way.

The definitions for shuffles follow the same pattern already established.
Given $[\pi]\in C(A)$ and $[\si]\in C(B)$ where $A\cap B=\emp$ we define their {\em cyclic shuffle set} to be
$$
[\pi]\shu[\si] =\{[\tau]\in C(A\uplus B) \mid
\text{$[\pi],[\si]$ are circular subwords of $[\tau]$}\}.
$$
Alternatively, there are rotations $\pi'$ and $\si'$ of $\pi$ and $\si$, respectively, which are both linear subwords of $\tau$.  To illustrate,
\begin{equation}
\label{13s24}
[13]\shu [24] = \{[1324], [1342], [1234], [1432], [1243], [1423]\}.    
\end{equation}
Call a cyclic permutation statistic $\cSt$ {\em cyclic shuffle compatible}
if for all quadruples $[\pi]$, $[\pi']$, $[\si]$, $[\si']$ with $\cSt[\pi]=\cSt[\pi']$, $\cSt[\si]=\cSt[\si']$, $\#\pi=\#\pi'$, and $\#\si=\#\si'$ we have
$$
\cSt([\pi]\shu[\si]) = \cSt([\pi']\shu[\si']).
$$
The aim of the present work is to study this concept.  In particular, we will prove the following theorem.
\begin{thm}
\label{cSt}
The statistics
$$
\cDes, \cdes, \cPk, \cpk
$$
are all cyclic shuffle compatible.
\end{thm}

The rest of this paper is structured as follows.
Rather than proving each of the cases of Theorem~\ref{cSt} in an ad hoc manner, we will develop a method for lifting linear shuffle compatibility results to cyclic ones.  This will be done in the next section.  Then in Section~\ref{a} we will apply our Lifting Lemma to each of the four statistics in turn.  Finally, we will end with a section of comments and a future direction for research.

%%%%%%%%%%%%%%%%%%%%%%%%%%%%%%%%

\section{The Lifting Lemma}
\label{llSec}

In this section we will provide a general method for proving cyclic shuffle compatibility results as corollaries of linear ones.  We will first set some notation. 
In the notations like $C(A)$ and $L(A)$ we will often drop the parentheses if they would cause double delimiters.  For example $L[n]$ is all the linear permutations of $[n]$.

If $A\sbs\bbP$ and $n\in\bbN$ then let
$$
A+n = \{a+n \mid a\in A\}.
$$
In particular,
$$
[m]+n = \{n+1,n+2,\ldots,n+m\}.
$$
We will also need to add integers to sets modulo some $m\in\bbP$.  So if $A\sbe[m]$ then define
$$
A+n\ (\Mod m) =\{a+n\ (\Mod m) \mid a\in A\}
$$
where the representatives are chosen to be in $[m]$.  For example
if $A=\{2,4,5\}$ and $m=6$ then
$$
A+3 = \{5,7,8\}
$$
and
$$
A+3\ (\Mod 6) = \{1,2,5\}.
$$

We will also need the notion of standardization.  Suppose $A,B\sbs\bbP$  with $\#A=\#B=n$ and 
$\pi=\pi_1\pi_2\ldots\pi_n\in L(A)$.  The {\em standardization of $\pi$ to $B$} is 
$$
\std_B \pi=f(\pi_1) f(\pi_2)\ldots f(\pi_n)\in L(B)
$$
where $f:A\ra B$ is the unique order-preserving bijection between $A$ and $B$.  To illustrate, if $A=\{2,4,6,7\}$ and $B=\{1,3,8,9\}$ then
$\std_B(4762) = 3981$.  If $B=[n]$ then we write just $\std \pi$ for $\std_{[n]} \pi$ and call this the {\em standardization} of $\pi$.  Standardization for cyclic permutations is defined in the analogous manner.  For example, $\std [79254] = [45132]$.

We first prove a result about cyclic descent statistics which is a cyclic analogue of one in the linear case~\cite{bs:bps}.
\begin{lem}
\label{cStLem}
Let $\cSt$ be a cyclic descent statistic.  For any four cyclic permutations $[\pi]$, $[\pi']$, $[\si]$, $[\si']$ such that 
$$
\std[\pi]=\std[\pi'] \qmq{and} \std[\si]=\std[\si']
$$
we have
$$
\cSt([\pi]\shu[\si]) = \cSt([\pi']\shu[\si']).
$$
\end{lem}
\begin{proof}
Since $\cSt$ is a cyclic descent statistic, its values only depend on the relative order of adjacent elements.  So it suffices to prove the case when
$$
[\pi]\uplus[\si] = [\pi']\uplus[\si'] = [m+n]
$$
where $m=\#\pi=\#\pi'$ and $n=\#\si=\#\si'$.  For simplicity, let $A=[m]$ and $B=[n]+m$.

We claim that it suffices to find, for any $\pi$ and $\si$ as in the previous paragraph, a $\cSt$-preserving bijection
$$
[\pi]\shu[\si] \ra\std_A[\pi]\shu\std_B[\si].
$$
For from this map and the hypothesis of the lemma we have
$$
\cSt([\pi]\shu[\si]) =\cSt(\std_A[\pi]\shu\std_B[\si])
=\cSt(\std_A[\pi']\shu\std_B[\si'])= \cSt([\pi']\shu[\si']).
$$

We will show the existence of this bijection by induction on the size of the set of what we will call out-of-order pairs
$$
O=\{(i,j) \in \pi\times \si \mid i>j\}. 
$$
If $\#O=0$ then $[\pi]\in C(A)$ and $[\si]\in C(B)$.  It follows that
$[\pi]=\std_A[\pi]$ and $[\si]=\st_B[\si]$ so the identity map will do.

For the induction step, let $\#O>0$.  Then there must be a pair $(i,i-1)\in O$.  Let $\pi''=(i-1,i)\pi$ and $\si''=(i-1,i)\si$ where $(i-1,i)$ is the transposition which exchanges $i-1$ and $i$.
We will be done if we can construct a $\cSt$ preserving bijection
$$
T_i:[\pi]\shu[\si] \ra [\pi'']\shu[\si''].
$$
This is because $\pi'',\si''$ have fewer out-of-order pairs and so, by induction, there is a $\cSt$-preserving bijection
$[\pi'']\shu[\si''] \ra\std_A[\pi'']\shu\std_B[\si'']$ which, when composed with $T_i$, will finish the construction.

Define $T_i$ by
$$
T_i[\tau]=
\begin{cases}
[(i-1,i)\tau] &\text{if $i-1,i$ are not cyclically adjacent in $[\tau]$,}\\
[\tau]  & \text{else.}
\end{cases}
$$
We must first check that $T_i$ is well defined in that $T_i[\tau]\in[\pi'']\shu[\si'']$.  This is true if $i-1$ and $i$ are not cyclically adjacent since $i-1$ and $i$ have been swapped in all three cyclic permutations involved.  If they are adjacent then the relative order of the elements of $[\tau]$ corresponding to $[\pi]$ and $[\pi'']$ are the same, and similarly for $[\si]$ and $[\si'']$.  So leaving $[\tau]$ fixed again gives a shuffle in the range.

Finally, we need to verify that $T_i$ is $\cSt$ preserving.  Since $\cSt$ is a descent statistic, it suffices to show that $T_i$ is $\cDes$ preserving.  Certainly this is true if $[\tau]$ is fixed.  And if it is not, then $i-1,i$ are not cyclically adjacent in $[\tau]$.  But switching $i-1$ and $i$ could only change a cyclic descent into a cyclic ascent or vice-versa if these two elements were adjacent.  So in this case $\cDes[(i-1,i)\tau]=\cDes[\tau]$ and we are done.
\end{proof}

As an example of the map $T_i$, consider the shuffle set in~\eqref{13s24}.  Here we can take $i=3$ since $3\in 13$ and $2\in 24$.  For $[\tau]=[1324]$ we have $2$ and $3$ cyclically adjacent so
$$
T_3[1324]=[1324]\in[12]\shu[34]
$$
as desired.  On the other hand, in $[1342]$ the $2$ and $3$ are not cyclically adjacent so
$$
T_3[1342]=[1243].
$$
Note that $[1243]\in[12]\shu[34]$ and 
$$
\cDes [1243] = \{\{\ \{1,2\},\ \{2,3\},\ \{3,4\},\ \{1,4\}\ \}\}
=\cDes [1342].
$$

We can use the previous lemma to drastically cut down on the number of cases which need to be checked to obtain cyclic shuffle compatibility.  In particular, the component permutations in the shuffles to be considered can be on consecutive intervals of integers.  And one can keep one component of the shuffle constant while the other varies.  
\begin{cor}
\label{cStCor}
Suppose that $\cSt$ is  cyclic descent statistic.  The following are equivalent.
\begin{enumerate}
    \item[(a)]  The statistic $\cSt$ is cyclic shuffle compatible.
    \item[(b)]  If $\cSt([\pi])=\cSt([\pi'])$ where $[\pi],[\pi']\in C[m]$ and
    $[\si]\in C([n]+m)$ for some $m,n\in\bbN$ then
    $$
    \cSt([\pi]\shu[\si]) =\cSt([\pi']\shu [\si]).
    $$
    \item[(c)] If $\cSt([\si])=\cSt([\si'])$ where $[\si],[\si']\in C([n]+m)$ and
    $[\pi]\in C[m]$ for some $m,n\in\bbN$ then
    $$
    \cSt([\pi]\shu[\si]) =\cSt([\pi]\shu [\si']).
    $$ 
\end{enumerate}
\end{cor}
\begin{proof}
We will prove the equivalence of (a) and (b) since the equivalence of (a) and (c) is similar.  Clearly (a) implies (b).  For the converse, let $\pi,\pi',\si,\si'$ be any four permutations satisfying the hypothesis of the cyclic shuffle compatible definition and let $m=\#\pi=\#\pi'$, $n=\#\si=\#\si'$.  Also let
$A=[m]$, $B=[n]+m$, $A'=[m]+n$, and $B'=[n]$.  Then, using Lemma~\ref{cStLem} and (b) alternately
\begin{align*}
 \cSt([\pi]\shu[\si])
&=\cSt(\std_A[\pi] \shu \std_B[\si])\\ 
&=\cSt(\std_A[\pi'] \shu \std_B[\si])\\
&=\cSt(\std_{A'}[\pi'] \shu \std_{B'}[\si])\\
&=\cSt(\std_{A'}[\pi'] \shu \std_{B'}[\si'])\\
&= \cSt([\pi']\shu[\si'])
\end{align*}
which is what we wished to prove.
\end{proof}

In order to prove the Lifting Lemma, we will need to define two functions.
For $i\in[n]$ we construct the {\em splitting map} $S_i:C[n]\ra L[n]$ as follows.  If $[\pi]\in C[n]$ then let $S_i[\pi]$ be the unique linear permutation in $[\pi]$ which starts with $i$.  For example,
$$
S_3[45132] = 32451.
$$

Also define the {\em maximum removal} map $M:C[n]\ra L[n-1]$ by first applying $S_n$ to $[\pi]$ and then removing the initial $n$.  To illustrate
$$
M[45132] = 1324.
$$
Note that $M$ is a bijection.  We similarly define
$M:C([n]+m)\ra L([n-1]+m)$ using $m+n$ in place of $n$, as in
$$
M[67354] = 3546.
$$

Finally, we say that $y$ is {\em between} $x$ and $z$ in $[\pi]$ if this is true of any linear permutation in $[\pi]$ where $x$ occurs to the left of $z$.
If $[\pi]\in C[m]$, $i\in[m]$, and $[\si]\in C([n]+m)$ define
$$
[\pi]\shu_i [\si] = \{[\tau]\in [\pi]\shu [\si] \mid \text{only elements of $\si$ are between $m+n$ and $i$ in $[\tau]$}\}.
$$
For example,
$$
[12]\shu[34] = 
\{[1234], [1243], [1324], [1423], [1342], [1432]\}
$$
with
$$
[12]\shu_1 [34] = 
\{[1234], [1243], [1324]\} 
$$
and
$$
[12]\shu_2 [34] = \{[1423], [1342], [1432]\}.
$$
Note that $[\pi]\shu [\si]$ is the disjoint union of the 
$[\pi]\shu_i [\si]$ for $i\in[m]$ since we can decompose the full shuffle set by finding the first element $i\in[m]$ that one encounters cyclically after $m+n$.

\begin{lem}[Lifting Lemma]
\label{ll}
Let $\cSt$ be a cyclic descent statistic and $\St$ be a shuffle compatible linear descent statistic such that the following conditions hold.
\begin{enumerate}
    \item[(a)] For any $[\tau],[\tau']$ of the same length
    $$
    \St(M[\tau])=\St(M[\tau']) \qmq{implies} \cSt[\tau]=\cSt[\tau'].
    $$
    \item[(b)] Given any $[\pi],[\pi']\in C[m]$ such that $\cSt[\pi]=\cSt[\pi']$, there exists a bijection $f:[m]\ra[m]$ such that for all $i$ and $j=f(i)$
    $$
    \St(S_i[\pi]) = \St(S_j[\pi']).
    $$
\end{enumerate}
Then $\cSt$ is cyclic shuffle compatible.
\end{lem}
\begin{proof}
By Corollary~\ref{cStCor}, it suffices to show that if $[\pi],[\pi']$ are as given in (b) and $\si\in C([n]+m)$ then $\cSt([\pi]\shu[\si]) =\cSt([\pi']\shu [\si])$.  The remarks preceding the lemma show that this reduces to proving
\begin{equation}
\label{shu_i}
   \cSt([\pi]\shu_i [\si]) =\cSt([\pi']\shu_j [\si]) 
\end{equation}
for all $i\in[m]$ and $j=f(i)$.

It follows that $\St(S_i[\pi]) = \St(S_j[\pi'])$ by (b).  So, since $\St$ is shuffle compatible, we have
$\St(S_i[\pi]\shu \si') = \St(S_j[\pi']\shu \si')$
where $\si'=M[\si]$.  Thus there is an $\St$-preserving bijection
$$
\th:S_i[\pi]\shu \si' \ra S_j[\pi']\shu \si'.
$$
This gives rise to a map
$$
\th':[\pi]\shu_i [\si]
\stackrel{M}{\ra} S_i[\pi]\shu \si'
\stackrel{\th}{\ra} S_j[\pi']\shu \si'
\stackrel{M^{-1}}{\ra} [\pi']\shu_j [\si]
$$
Since this is a bijection, to show~\eqref{shu_i} it suffices to 
prove that $\th'$ is $\cSt$-preserving.  So
take $[\tau]\in [\pi]\shu_i [\si]$ and $[\tau']=\th'[\tau]$.
Then $M[\tau'] = \th\circ M[\tau]$.  Since $\th$ is $\St$ preserving, we have $\St M[\tau'] = \St M[\tau]$.  But then hypothesis (a) implies that $\cSt[\tau]=\cSt[\tau']$ which is what we wished to prove.
\end{proof}

%%%%%%%%%%%%%%%%%%%%%%%%%%%%%%%%

\section{Applications}
\label{a}

The hypotheses in the Lifting Lemma may seem strange at first glance.  But they can be quite easy to verify, making it a useful tool.  We will see four instances of this by proving Theorem~\ref{cSt} using its aid.

\begin{thm}
\label{cDes}
The statistic $\cDes$ is cyclic shuffle compatible.
\end{thm}
\begin{proof}
We will verify the hypotheses of Lemma~\ref{ll} using $\St=\Des$.  For (a), suppose that $\#\tau=\#\tau'=n$ and 
\begin{equation}
\label{Des=}  
\Des(M[\tau])=\Des(M[\tau']).
\end{equation}
Let us assume that $\tau,\tau'$ were chosen from their cyclic equivalence class so that $\tau_1=\max\tau$ and similarly for $\tau'$.  Then $M[\tau]=\tau_2\tau_3\ldots\tau_n$.  And by the choice of $\tau_1$ we see that
\begin{equation}
\label{U1}  
\Des\tau=\{1\}\uplus(\Des M[\tau]+1).
\end{equation}
  Since the same statement holds with $\tau'$ in place of $\tau$ and~\eqref{Des=} holds, we have
$\Des\tau=\Des\tau'$.  But $\Des\tau$ is one set in the multiset $\cDes[\tau]$, and the others are gotten by adding each $i\in[n]$ to all elements of $\Des\tau$ modulo $n$.  The same being true of $\cDes[\tau']$ shows that
$\cDes[\tau]=\cDes[\tau']$ as desired.

For (b), we are given $[\pi],[\pi']\in C[m]$ with $\cDes[\pi]=\cDes[\pi']$ and must construct the necessary bijection.  From this assumption, we can choose $\pi$ and $\pi'$ such that 
\begin{equation}
\label{cDes=}
\cDes\pi=\cDes\pi'=A.   
\end{equation}
for some set $A$.  Define $f:[m]\ra[m]$ by
$f(\pi_k)=\pi_k'$ for all $k\in[m]$.  Let $i=\pi_k$ and $j=\pi_k'$.  To show that $\Des(S_i[\pi])=\Des(S_j[\pi'])$ there are two cases depending on whether $k-1\in A$ or not, where $k-1$ is taken modulo $m$. 

If $k-1\not\in A$ then $\pi_k,\pi_k'$ are not the second elements in cyclic descents of their respective  permutations.  From this and equation~\eqref{cDes=} it follows that
$$
\Des(S_i[\pi]) = A - k + 1\ (\Mod m) = \Des(S_j[\pi']).
$$
If $k-1\in A$ then the same equalities hold with $A$ replaced by $A'$ which is $A$ with $k-1$ removed.  The completes the proof of (b) and of the theorem.
\end{proof}

\begin{thm}
\label{cdes}
The statistic $\cdes$ is cyclic shuffle compatible.
\end{thm}
\begin{proof}
We proceed as in the previous proof with $\St=\des$.  For (a) we  assume $\des(M[\tau])=\des(M[\tau'])$.  But from equation~\eqref{U1} we see that $\des\tau=\des(M[\tau]) - 1$ and the same is true for $\tau'$.   Combining this with our initial assumption gives $\des\tau=\des\tau'$.  But $\tau,\tau'$ begin with their largest element so that
$$
\cdes[\tau]=\des\tau=\des\tau'=\cdes[\tau']
$$
which is what we wished to show.

To prove (b), we are given $\cdes[\pi]=\cdes[\pi']$.  As in the preceding paragraph, choose $\pi,\pi'$ to begin with their largest elements so that $\cdes[\pi]=\des\pi$ and similarly for $\pi'$.  Since $\des\pi=\des\pi'$ there is a bijection $\th:\Des\pi\ra\Des\pi'$.  Extend this map to $\th:[m]\ra[m]$ by using any bijection between the complements of $\Des\pi$ and $\Des\pi'$.  The proof that $\th$ has the desired property is now similar to that for $\Des$ and so is left to the reader.
\end{proof}

\begin{thm}
\label{cPk}
The statistic $\cPk$ is cyclic shuffle compatible.
\end{thm}
\begin{proof}
This proof parallels the one for $\cDes$, so we will only mention the highlights.  For (a) one sees that $\Pk\tau = \Pk M[\tau]+1$ since $\tau_1=\max\tau$ is not a peak in $\tau$ and is removed in $M[\tau]$.  But this still implies that $\Pk\tau=\Pk\tau'$ and the rest of this part of the demonstration goes through.

For (b), the map $\th$ is constructed in exactly the same way using 
$A=\cPk\pi=\cPk\pi'$.  The only difference with the remaining part of the proof is that there are two subcases when $k-1\not\in A$ depending on whether $k-2\in A$ or not.  If $k-2\in A$ then one loses the peak which was at $\pi_{k-2}$ in $S_i[\pi]$.  On the other hand, the peak set stays the same modulo rotation if $k-2\not\in A$.  But since the analogous statements hold for $\pi'$, the manipulations in these subcases are as before.
\end{proof}

The proof of the next result is based on the proof of Theorem~\ref{cPk} in much the same way that the demonstrations of Theorems~\ref{cDes} and~\ref{cdes} are related.  So the details are left to the reader.
\begin{thm}
\label{cpk}
The statistic $\cpk$ is cyclic shuffle compatible.\hqed
\end{thm}

Lest it appear that cyclic shuffle compatibility follows exactly the same lines as the linear case, let us point out a place where they differ.  
Define a {\em factor} of $\pi$ to be a subsequence of consecutive elements.
A {\em birun} of $\pi$ is a maximal monotone factor.
Let $\bru\pi$ be the number of biruns of $\pi$.  For example,
$\bru 125346= 3$ because of the biruns $125$, $53$, and $346$.
It is easy to see~\cite{gz:sps} that the birun statistic is not linearly
shuffle compatible.
A {\em cyclic birun} of $[\pi]$ is defined in the obvious manner and denoted $\cbru[\pi]$.  Returning to our example,
$\cbru [125346]= 4$ because of the biruns above and $61$.
\begin{thm}
The statistic $\cbru$ is cyclic shuffle compatible.
\end{thm}
\begin{proof}
Clearly cyclic biruns always begin  at a cyclic peak and end at a cyclic valley, or vice-versa.  So $\cbru[\pi]=2\cpk[\pi]$ and the result follows from 
Theorem~\ref{cpk}.
\end{proof}

%%%%%%%%%%%%%%%%%%%%%%%%%%%%%%%%

\section{Remarks and an open question}
\label{roq}

\subsection{Cyclic pattern avoidance and statistics}

We say that $[\si]\in C[n]$ {\em contains the pattern} $[\pi]\in C[m]$ if there is some subsequence $[\si']$ of $[\si]$ such that $\std[\si']=[\pi]$.  Otherwise $[\si]$ {\em avoids} $[\pi]$.  For any set $[\Pi]$ of cyclic permutations we let
$$
\Av_n[\Pi]=\{[\si]\in C[n] \mid 
\text{$[\si]$ avoids $[\pi]$ for all $[\pi]\in[\Pi]$}\}.
$$
Callan calculated $\#\Av_n[\Pi]$ when $[\Pi]$ consists of a single element of $C[4]$.  Gray, Lanning, and Wang then contributed work on cyclic packing of patterns~\cite{GLW:pcc1} and patterns in colored cyclic permutations~\cite{GLW:pcc2}.  Most recently, the authors of this article determined $\#\Av_n[\Pi]$ for all 
$[\Pi]\sbe C[4]$.  They also calculated the generating function
$$
D_n([\Pi];q) =\sum_{[\si]\in\Av_n[\Pi]} q^{\cdes[\si]}.
$$
for the same sets of patterns.

\subsection{Cyclic maj?}

One famous linear statistic which we have not discussed so far is the {\em major index} of $\pi$ which is
$$
\maj\pi=\sum_{i\in\Des\pi} i.
$$
 Stanley's theory of $P$-partitions~\cite{sta:osp} shows that
 if $\#\pi=m$ and $\#\si=n$ then
$$
\sum_{\tau\in\pi\shu\si} q^{\maj\tau}
=q^{\maj\pi+\maj\si}\gauss{m+n}{m}_q
$$
where the last factor is a $q$ binomial coefficient.  Note that this is a $q$-analogue of the fact that $\#(\pi\shu\si)=\binom{m+n}{m}$.

It is easy to see that
\begin{equation}
\label{ShuCard}
  \#([\pi]\shu[\si]) = (m+n-1) \binom{m+n-2}{m-1}.  
\end{equation}
Indeed, write each $[\tau]$ in the shuffle so that it starts with  $\pi_1$.  Then there are $m+n-1$ more slots to be filled and $m-1$ of them must be for the elements $\pi_2\pi_3\ldots\pi_m$ in that order.  The remaining slots can be filled with any of the $n$ rotations of $\si$ for a total of $n\binom{m+n-1}{m-1}$ choices.  This is equivalent to the formula above.  In~\cite{agrr:cqf} the following question was raised.
\begin{question}
Is there a cyclic analogue of $\maj$ which has nice properties such as giving a $q$-analogue of equation~\eqref{ShuCard}?
\end{question}
More generally, one could try to find shuffle compatible cyclic analogues of other linear statistics.

%\bibliographystyle{plain}

%\nocite{*}
%\bibliographystyle{abbrvnat}

\nocite{*}
\bibliographystyle{alpha}

\newcommand{\etalchar}[1]{$^{#1}$}

\end{document}